\documentclass[11pt]{amsart}

\usepackage{amssymb,amsmath,amsthm,amscd}

\input{amssym.def}
\newcommand{\bt}{\begin{Theorem}}
\newcommand{\et}{\end{Theorem}}
\newcommand{\bi}{\begin{itemize}}
\newcommand{\ei}{\end{itemize}}
\newcommand{\bea}{\begin{eqnarray}}
\newcommand{\eea}{\end{eqnarray}}

\theoremstyle{plain}

\newtheorem{theorem}{\sc Theorem}[section]
\newtheorem{Lemma}{\sc Lemma}[section]
\newtheorem{Proposition}{\sc Proposition}[section]

\newtheorem{Corollary}{\sc Corollary}[section]

\theoremstyle{definition}

\theoremstyle{remark}

\newcommand{\be}{\begin{equation}}
\newcommand{\ee}{\end{equation}}

\title{Multiplicity-free homogeneous operators in the Cowen-Douglas class}

\author{Adam Kor\'{a}nyi}
\address{Lehman College\\
The City University of New York\\
Bronx, NY 10468
}
\email{adam.koranyi@lehman.cuny.edu}
\author{Gadadhar Misra}
\address{Department of Mathematics \\
Indian Institute of Science\\
Bangalore 560 012
 }
\email{gm@math.iisc.ernet.in}
\thanks{This research was supported in part by a DST -
NSF S\&T Cooperation Programme}
\begin{document}
\begin{abstract}

In a recent paper, the authors have constructed a large class of
operators in the Cowen-Douglas class {Cowen-Douglas class}
 of the unit disc $\mathbb D$
which are {\em homogeneous} with respect to the action of the group
M\"{o}b -- the M\"{o}bius group consisting of bi-holomorphic
automorphisms of the unit disc $\mathbb{D}$. The {\em associated
representation} for each of these operators is {\em multiplicity
free}.\index{multiplicity free} Here we give a different
independent construction of all
homogeneous operators in the Cowen-Douglas class with multiplicity
free associated representation and verify that they are exactly the
examples constructed previously.

\end{abstract}
\maketitle
The homogeneous operators form a class of bounded operators $T$ on a
Hilbert space ${\mathcal H}$.  The operator $T$ is said to be {\em
homogeneous} if its spectrum is contained in the closed unit disc
and for every M\"{o}bius transformation $g$ the operator $g(T)$,
defined via the usual holomorphic functional calculus, is unitarily
equivalent to $T$. To every homogeneous irreducible operator $T$
there corresponds an {\em associated}  unitary representation $\pi$
of the  universal covering group $\tilde{G}$ of the M\"{o}bius group
$G$:
$$\pi(\hat{g})^*\, T \, \pi(\hat{g}) =  (p\hat{g})\,(T),\: \hat{g} \in \tilde{G},$$
where $p:\tilde{G} \to G$ is the natural homomorphism.  In the paper
\cite{KM}
(see also \cite{BisGM}),
it was shown that  each
homogeneous operator $T$, not necessarily irreducible, in ${\rm
B}_{m+1}(\mathbb{D})$ admits an associated representation. The
representations of $\tilde{G}$ are quite well-known, but we are
still far from a complete description of the homogeneous operators.
In the recent paper \cite{KM},
the following theorem was proved.

\begin{theorem} \label{thkm1}
For any positive real number $\lambda > m/2$, $m\in \mathbb{N}$ and
an $(m+1)$ - tuple of positive reals ${\boldsymbol \mu}=(\mu_0,
\mu_1, \ldots , \mu_m)$ with $\mu_0=1$, there exists a reproducing
kernel $K^{(\lambda, {\boldsymbol \mu})}$ on the unit disc such that
the adjoint of the multiplication operator $M^{(\lambda,{\boldsymbol
\mu})}$ on the corresponding Hilbert space $\mathbf
A^{(\lambda,{\boldsymbol \mu})}(\mathbb D)$ is homogeneous. The
operators $(M^{(\lambda,{\boldsymbol \mu})})^*$ are in the
Cowen-Douglas class\index{Cowen-Douglas class} ${\rm B}_{m+1}(\mathbb{D})$, irreducible and
mutually inequivalent.
\end{theorem}

In the paper \cite{KM},
we have presented the operators
$M^{(\lambda,{\boldsymbol \mu})}$ in as elementary a way as
possible, but this presentation hides the natural ways in which
these operators can be found to begin with. Here we will describe
another independent construction of the operators
$M^{(\lambda,{\boldsymbol \mu})}$. We will also give an exposition
of some of the fundamental background material. Finally, we will
prove that if $T$ is an irreducible homogeneous operator in ${\rm
B}_{m+1}(\mathbb{D})$ whose associated representation is
multiplicity free\index{multiplicity free}
then, up to equivalence, $T$ is the adjoint of of
the  multiplication operator $M^{(\lambda, {\boldsymbol \mu})}$ for
some $\lambda > m/2$ and ${\boldsymbol \mu} \geq 0$.

\section{Background material}
Although, we intend to discuss homogeneous operators in the
Cowen-Douglas class\index{Cowen-Douglas class} $\mathrm B_n(\mathbb D)$, the material below is
presented in somewhat greater generality.  Here we discuss commuting
tuples of operators in the Cowen-Douglas class ${\mathrm
B}_n(\mathcal D)$ for some bounded open connected set $\mathcal
D\subseteq \mathbb C^m$.  The unitary equivalence class of a
commuting tuple in ${\mathrm B}_n(\mathcal D)$ is in one to one
correspondence with a certain class of holomorphic Hermitian vector
bundles (hHvb) on $\mathcal D$
\cite{C-D}.
These are distinguished
by the property, among others, that the Hermitian structure on the
fibre at $w\in \mathcal D$ is induced by a reproducing kernel $K$.
It is shown in \cite{C-D}
that the corresponding operator can be
realized as the adjoint of the   commuting tuple multiplication
operator $\mathbf M$ on the Hilbert space $\mathcal H$ of
holomorphic functions with reproducing kernel $K$.

Start with a Hilbert space $\mathcal H$ of $\mathbb C^n$ - valued
holomorphic functions on a bounded open connected set $\mathcal D
\subseteq \mathbb C^m$.  Assume that the Hilbert space $\mathcal H$
contains the set of vector valued polynomials and that these form a
dense subset in $\mathcal H$.  We also assume that there is a
reproducing kernel $K$ for $\mathcal H$.  We use the notation
$K_w(z):= K(z,w)$.

Recall that a positive definite kernel $K: \mathcal{D} \times \mathcal{D}
\to \mathbb C^{n \times n}$ on $\mathcal{D}$ defines an inner product on
the linear span of $\{K_w(\cdot)\xi : w \in \mathcal D, \xi \in
\mathbb C^n\} \subseteq {\rm Hol}(\mathcal{D}, \mathbb{C}^n)$ by the
rule
$$
\langle K_w(\cdot) \xi,K_u(\cdot) \eta\rangle = \langle K_w(u)
\xi,\eta\rangle,\: \xi,\eta \in \mathbb C^n.
$$
(On the right hand side $\langle , \rangle$ denotes the inner
product of $\mathbb C^n$.  We denote by $\varepsilon_1, \ldots,
\varepsilon_n$ the natural basis of $\mathbb C^n$.)  The completion
of this subspace is then a Hilbert space $\mathcal H$ of holomorphic
functions on $\mathcal{D}$ (cf. \cite{A}) in which
the set of vectors $\{K_w: w \in \mathcal D\}$ is dense.
The kernel $K$ has the reproducing property, that is,
$$
\langle f,K_w \xi\rangle = \langle f(w),\xi\rangle,\: f \in \mathcal
H,\,w\in \mathcal D,\, \xi \in \mathbb C^m.
$$
Now, for $1 \leq i \leq m$, we have
$$M^*_i K_w \xi  = \bar{w}_i K_w \xi,\, w \in \mathcal D,\,\mbox{where}
\, \big (M_i f\big )(z)  =  z_i  f(z),\, f\in \mathcal H$$
and $\{K_w \varepsilon_i\}_{i=1}^n$ is a basis for $\cap_{i=1}^m
\ker (M_i - w_i)^*$, $w\in \mathcal D$.

The joint kernel of the commuting $m$ - tuple $\mathbf M^* = (M_1^*,
\ldots , M_m^*)$, which we assume to be bounded, then has dimension
$n$.   The map $\sigma_i: w \mapsto K_{\bar{w}} \varepsilon_i$, $w
\in \bar{\mathcal D}$, $1\leq i \leq n$, provides a trivialization
of the corresponding bundle $E$ of Cowen - Douglas (cf.
\cite{C-D}).
Here $\bar{\mathcal D}:=\{z \in \mathbb C^m \mid \bar{z} \in
\mathcal D \}$).

On the other hand, suppose we start with an abstract Hilbert space
$\mathcal H$ and a $m$-tuple  of commuting operators $\mathbf T=
(T_1, \ldots , T_m)$  in the Cowen - Douglas class\index{Cowen-Douglas class} $\mathrm
B_n(\mathcal D)$.  Then  we have a holomorphic Hermitian vector
bundle $E$ over $\mathcal D$ with the fibre $E_w = \cap_{i=1}^n \ker
(T_i - w_i)$ at $w\in \mathcal D$. Following
\cite{C-D}, one
associates to this a reproducing kernel Hilbert space $\hat{\mathcal
H}$ consisting of holomorphic functions on $\bar{\mathcal D}$ as
follows.  Take a holomorphic trivialization $\sigma_i: \mathcal D
\to \mathcal H$ with $\sigma_i(w),\,1\leq i \leq n$, spanning $E_w$.
For $f\in \mathcal H$, define $\hat{f}_j(w) := \langle
f,\sigma_j(\bar{w})\rangle_\mathcal H$, $w\in \bar{\mathcal D}$. Set
$\langle \hat{f},\hat{g}\rangle_{\hat{\mathcal H}} := \langle
f,g\rangle_\mathcal H$. The function $K_w \varepsilon_j:=
\widehat{\sigma_j(\bar{w})}$ then serves as the reproducing kernel
for the Hilbert space $\hat{\mathcal H}$.  Note that
\begin{eqnarray*}
\langle K_w(z)\varepsilon_j,\varepsilon_i\rangle_{\mathbb C^n} &=&
\langle K_w\varepsilon_j,K_z\varepsilon_i\rangle_{\hat{\mathcal H}}\\
& = & \langle \widehat{\sigma_j(\bar{w})},\widehat{\sigma_i(\bar{z})}
\rangle_{\hat{\mathcal H}}\\
&= & \langle \sigma_j(\bar{w}),\sigma_i(\bar{z})\rangle_{\mathcal
H},\,z,w \in \bar{\mathcal D}.
\end{eqnarray*}

If one applies this construction to the case where $\mathcal H$ is a
Hilbert space of holomorphic functions on $\mathcal D$, possesses a
reproducing kernel, say $K$, and the operator $\mathbf M^*$ is in
$\mathrm B_n(\bar{\mathcal D})$ then using the trivialization
$\sigma_i(w) = K_{\bar{w}} \varepsilon_i$, $w\in \bar{\mathcal D}$
for the bundle $E$ defined on $\bar{\mathcal D}$, the reproducing
kernel for $\hat{\mathcal H}$ is
\begin{eqnarray*}
\langle K_w(z)\varepsilon_j,\varepsilon_i\rangle_{\mathbb C^n} &=&
\langle K_w \varepsilon_j,K_z \varepsilon_j\rangle_{\mathcal H}\\
&= & \langle \sigma_j(\bar{w}),\sigma_i(\bar{z})\rangle_{\mathcal H}\\
&= & \langle K_w\varepsilon_j,K_z\varepsilon_i\rangle_{\hat{\mathcal H}},
\,z,w \in \mathcal D.
\end{eqnarray*}
Thus $\mathcal H = \hat{\mathcal{H}}$.

Let $G$ be a Lie group
acting transitively on the domain $\mathcal{D}\subseteq
\mathbb{C}^d$. Let ${\mathrm GL}(n,\mathbb{C})$ denote the set of
non-singular $n\times n$ matrices over the complex field
$\mathbb{C}$. We start with a multiplier $J$, that is, a smooth
family of holomorphic maps $J_g:\mathcal{D} \to \mathbb{C}^{n\times n}$
satisfying the cocycle relation
\begin{equation} \label{cocycle}
J_{gh}(z)= J_h(z)J_g(h\cdot z),~~\mbox{for all~~} g,h\in
G,~z\in\mathcal{D},
\end{equation}
Let ${\rm Hol}(\mathcal{D}, \mathbb{C}^n)$ be the linear space
consisting of all holomorphic functions on $\mathcal{D}$ taking values
in $\mathbb{C}^n$. We then obtain a natural (left) action $U$ of the
group $G$ on ${\rm Hol}(\mathcal{D}, \mathbb{C}^n)$:
\begin{equation}
  \label{eq:grp action}
  (U_{g}f)(z) = J_{g^{-1}}(z)f(g^{-1} \cdot z),~f\in
  {\rm Hol}(\mathcal{D},\mathbb{C}^n),~z\in \mathcal{D}.
\end{equation}

Let $\mathbb{K}\subseteq G$ be the compact subgroup which is the
stabilizer of $0$.  For $h,k$ in $\mathbb{K}$, we have
$J_{kh}(0)=J_h (0)J_k (0)$ so that $k\mapsto J_k (0)^{-1}$ is a
representation of $\mathbb{K}$ on $\mathbb{C}^n$.

As in \cite{KM},
 we say that if a reproducing kernel $K$ transforms according to the rule
\begin{equation} \label{eq:trans rule}
J(g,z)K(g(z), g(\omega)) J(g,\omega)^* = K(z,\omega)
\end{equation}
for all $g \in \tilde{G}$; $z,\omega \in \mathbb D$,  then $K$ is
{\em quasi-invariant}.

\begin{Proposition}[\cite{KM},
Proposition 2.1] \label{pkm1}
Suppose ${\mathcal H}$ has a reproducing kernel $K$.  Then $U$ defined
by \eqref{eq:grp action} is a unitary representation if and only if
$K$ is quasi-invariant.
\end{Proposition}

Let $g_z$ be an element of $G$ which maps $0$ to $z$, that is $g_z
\cdot 0 = z$.

For quasi-invariant $K$ we have
\begin{equation}  \label{eq:homker def}
  K(g_z\cdot 0, g_z\cdot 0) = (J_{g_z}(0))^{-1}K(0,0)(J_{g_z}(0)^*)^{-1},
\end{equation}
which shows that $K(z,z)$ is uniquely determined by $K(0,0)$.  For
each $z$ in $\mathcal D$, the positive definite matrix $K(z,z)$
gives the Hermitian structure of our vector bundle.

Given any positive definite matrix $K(0,0)$ such that
\begin{equation} \label{J_kCommutation}
J_k(0)^{-1}K(0,0) = K(0,0)J_k(0)^*\: \mbox{for all}\: k \in \mathbb
K,
\end{equation}
that is, the inner product $\langle K(0,0)\cdot \mid \cdot \rangle$
is invariant under $J_k(0)$, \eqref{eq:homker def} defines a
Hermitian structure on the homogeneous vector bundle determined by
$J_g(z)$. In fact, $K(z,z)$, for any $z\in \mathcal D$ is well
defined, because if    $g_z^\prime$ is another element of $G$ such
that $g_z^\prime\cdot 0=z$ then $g_z^\prime =g_z k$ for some $k\in
\mathbb{K}$. Hence
\begin{eqnarray*}
K(g_z^\prime\cdot 0, g_z^\prime\cdot 0) &=&  K(g_z k\cdot 0, g_z k\cdot 0) \\ &=&(J_{g_zk} (0))^{-1}K(0,0)(J_{g_zk} (0)^*)^{-1}\\
  &=& \big ( J_k(0)J_{g_z}(k\cdot 0) \big )^{-1}K(0,0)\big (
  J_{g_z}(k\cdot 0)^*  J_k(0)^*\big )^{-1} \\
  &=& (J_{g_z}(0))^{-1} (J_k(0))^{-1} K(0,0)
  (J_k(0)^*)^{-1}(J_{g_z}(0)^*)^{-1}\\
  &=& (J_{g_z}(0))^{-1} K(0,0)(J_{g_z}(0)^*)^{-1}\\
  &=& K(g_z\cdot 0, g_z\cdot 0)
\end{eqnarray*}

This gives a good overview of all the Hermitian structures of a
homogeneous holomorphic vector bundle.  But not all such bundles
arise from a reproducing kernel.  Starting with a positive matrix
satisfying \eqref{J_kCommutation},  \eqref{eq:homker def} gives us
$K(z,z)$, but there is no guarantee (and is false in general) that
$K(z,z)$ extends to a positive definite kernel on $\mathcal D \times
\mathcal D$.  It is, however, true that if there is such an
extension then it is uniquely determined by $K(z,z)$ (because
$K(z,w)$ is holomorphic in $z$ and antiholomorphic in $w$).

This leaves us with the following possible strategy for finding the
homogeneous operators in the Cowen - Douglas class.\index{Cowen-Douglas class}  Find all
multipliers, (i.e., holomorphic homogeneous vector bundles (hhvb))
such that there exists $K(0,0)$ satisfying \eqref{J_kCommutation}
and consider all such $K(0,0)$. Then determine which of the $K(z,z)$
obtained form \eqref{eq:homker def}  extends to a positive definite
kernel on $\mathcal D \times \mathcal D$.  Then check if the
multiplication operator is well-defined and bounded on the
corresponding Hilbert space.

Let $\mathcal{H}$ be a Hilbert space consisting of $\mathbb{C}^n$ -
valued holomorphic functions on some domain $\mathcal D$ possessing
a reproducing kernel $K$. The sections of the corresponding
holomorphic Hermitian vector bundle defined on $\mathcal D$ have
many different realizations. The connection between two of these is
given by a $n\times n$ invertible matrix valued holomorphic function
$\varphi$ on $\mathcal{D}$. For $f\in \mathcal{H}$, consider the map
$\Gamma_\varphi: f\mapsto \tilde{f}$, where $\tilde{f}(z)=
\varphi(z)f(z)$.  Let $\tilde{\mathcal{H}}=\{\tilde{f}: f\in
\mathcal{H}\}$.  The requirement that the map $\Gamma_\varphi$ is
unitary, prescribes a Hilbert space structure for the function space
$\tilde{\mathcal{H}}$.  The reproducing kernel for
$\tilde{\mathcal{H}}$ is easily calculated
\begin{equation} \label{rk for tilde}
\tilde{K}(z,w) = \varphi(z) K(z,w) \varphi(w)^*.
\end{equation}
It is also easy to verify that $\Gamma_\varphi M \Gamma_\varphi^*$
is the multiplication operator~ $M: \tilde{f} \mapsto z \tilde{f}$
on the Hilbert space $\tilde{\mathcal{H}}$.  Suppose we have a
unitary representation $U$ given by a multiplier $J$ acting on
$\mathcal{H}$ according to (\ref{eq:grp
  action}).   Transplanting this action to $\tilde{\mathcal{H}}$ under
the isometry $\Gamma_\varphi$, it becomes
$$
\big (\tilde{U}_{g^{-1}} \tilde{f}\big )(z) =
\tilde{J}_g(z)\tilde{f} (g\cdot z),
$$
where the new multiplier $\tilde{J}$ is given in terms of the
original multiplier $J$ by
$$
\tilde{J}_g(z) = \varphi(z)J_g(z) \varphi(g\cdot z)^{-1}.
$$
Of course, now $\tilde{K}$ transforms according to (\ref{eq:trans
rule}), with the aid of $\tilde{J}$. If we want, we can now ensure
that, by passing from $\mathcal{H}$ to an appropriate
$\tilde{\mathcal{H}}$, $\tilde{K}(z,0) \equiv 1$. We merely have to
set $\varphi(z) = K(0,0)^{1/2}K(z,0)^{-1}$. Thus the reproducing
kernel $\tilde{K}$ is almost unique.  The only freedom left is to
multiply $\varphi(z)$ by a constant unitary $n\times n$ matrix. Once
the kernel is normalized, we have
$$
 J_k(z) = J_k(0), ~~z\in \mathbb{D},~k\in \mathbb{K}.
$$
In fact,
$$
I = K(z,0)=J_k(z)K(k\cdot z, 0)J_k(0)^* = J_k(z)J_k(0)^{-1}
$$
and the statement follows. Therefore, once the kernel $K$ is
normalized, we have
$$
\big ( U_{k^{-1}} f \big ) (z) = J_k(0)f(k\cdot z),~~k\in
\mathbb{K}.
$$


Given a multiplier $J$, there is always the following method for
constructing a Hilbert space with a quasi-invariant Kernel $K$
transforming according to \eqref{eq:homker def}.  We look for a
functional Hilbert space possessing this property among the weighted
$L^2$ spaces of holomorphic functions on $\mathcal{D}$.  The norm on
such a space is
\begin{equation} \label{inv norm}
 \|f\|^2 = \int_\mathcal{D}f(z)^* Q(z)f(z) dV(z)
\end{equation}
with some positive matrix valued function $Q(z)$. Clearly, this
Hilbert space possesses a reproducing kernel $K$.  The condition
that $U_{g^{-1}}$ in (\ref{eq:grp action}) is unitary is
\begin{eqnarray*}
\lefteqn{\!\!\!\!\!\!\!\!\!\!\!\!\!\!\!\!\!\!\!\!\!\!\int_\mathcal{D}f(g\cdot z)^* J_g^*(z)Q(z)J_g(z)f(g\cdot z) dV(z)= \int_\mathcal{D}f(w)^* Q(w)f(w)dV(w)}\\
&\qquad\qquad =&  \int_\mathcal{D}f(g\cdot z)^* Q(g\cdot z) f(g\cdot z)
\left |\frac{\partial(g\cdot z)}{\partial(z)} \right |^2 dV(z),
\end{eqnarray*}
that is,
\begin{equation} \label{inv quadratic}
Q(g\cdot z)= J_g(z)^*Q(z)J_g(z) \left |\frac{\partial(g\cdot
z)}{\partial(z)} \right |^{-2},
\end{equation}
which is equation (\ref{eq:trans rule}) with $J_g(z)$ replaced by
$\frac{\partial(g\cdot z)}{\partial(z)}J_g(z)^{*-1}$.

Given the multiplier $J_g(z)$, $Q(z)$  is again determined by
$Q=Q(0)$, and (just as in the case of $K(0,0)=A$) it must be a
positive matrix commuting with all $J_k(0)$, $k\in \mathbb{K}$. (It
is assumed that each $J_k(0)$ is unitary).

In this way, we can construct many examples of homogeneous operators
in $\mathrm B_n(\mathcal D)$ but not all.

Even, not all the the homogeneous operators in $\mathrm B_1(\mathbb
D)$ come from this construction.  There is a homogeneous operator in
the class $\mathrm B_1(\mathbb D)$ corresponding to the multiplier
$J(g,z) = (g^\prime(z))^{ \lambda}$, $\lambda \in \mathbb R$ exactly
when $\lambda >0$.  The reproducing kernel is $K(z,w) = (1-z
\bar{w})^{-2 \lambda}$. But such an operator arises from the
construction outlined above only if $\lambda \geq 1/2$.

Never the less, the homogeneous operators constructed in the manner
described above are of interest since they happen to be exactly the
subnormal homogeneous operators in this class (cf.
\cite{matball}).

\section{Computation of the multipliers for the unit disc}

In the case of ${\mathrm B}_n(\mathbb D)$, it is shown in
\cite{KM} 
that the bundle corresponding to a homogeneous Cowen-Douglas
operator admits an action of  the covering group $\tilde{G}$ of the
group $G=\mbox{M\"{o}b}$ via unitary bundle maps.  This suggests the
strategy of first finding all the homogeneous holomorphic Hermitian
vector bundles (a problem easily solved by known methods) and then
determining which of these correspond to an operator in the
Cowen-Douglas\index{Cowen-Douglas class} class.

We are going to use the method of holomorphic induction.  For this,
first we describe some basic facts and fix our notation.  We follow
the notation of \cite{Sig} which we will
use as a reference.

The Lie algebra $\mathfrak g$ of $\tilde{G}$ is spanned by $X_1=
\frac{1}{2} \left(
              \begin{array}{cc}
                0 & 1 \\
                1 & 0 \\
              \end{array}
            \right)$, $X_0 = \frac{1}{2} \left(
                                           \begin{array}{cc}
                                             i & 0 \\
                                             0 & -i \\
                                           \end{array}
                                         \right)
            $ and $Y= \frac{1}{2}
\left(
  \begin{array}{cc}
    0 & -i \\
    i & 0 \\
  \end{array}
\right)$.  The subalgebra $\mathfrak k$ corresponding to
$\tilde{\mathbb K}$ is spanned by $X_0$. In the complexified Lie
algebra $\mathfrak g^{\mathbb C}$, we  mostly use the complex basis
$h,x,y$ given by
\begin{eqnarray*}
h &=& - i X_0 = \frac{1}{2} \left(
                              \begin{array}{cc}
                                1 & 0 \\
                                0 & -1 \\
                              \end{array}
                            \right)
\\
x &=& X_1+i Y = \left(
                  \begin{array}{cc}
                    0 & 1 \\
                    0 & 0 \\
                  \end{array}
                \right)
\\
y &=& X_1 - i Y = \left(
                    \begin{array}{cc}
                      0 & 0 \\
                      1 & 0 \\
                    \end{array}
                  \right)
\end{eqnarray*}

We write $G^{\mathbb C}$ for the  (simply connected group) ${\mathrm
SL}(2,\mathbb C)$. Let $G_0={\mathrm SU}(1,1)$ be the subgroup
corresponding to $\mathfrak g$. The group $G^{\mathbb C}$ has the
closed subgroups $\mathbb K^\mathbb C = \Big \{ \Big (\begin{matrix}
z & 0\\ 0 & \tfrac{1}{z}
\end{matrix} \Big ):z \in \mathbb{C}, z\not = 0 \Big \}$, $P^+=\Big \{ \Big (\begin{matrix} 1 & z\\ 0 & 1 \end{matrix} \Big ):z\in \mathbb{C} \Big \}$,
$P^- = \Big \{ \Big (\begin{matrix} 1 & 0\\ z & 1 \end{matrix} \Big
): z\in \mathbb{C} \Big \}$; the corresponding Lie algebras
$\mathfrak k^\mathbb C = \Big \{ \Big (\begin{matrix} c & 0\\ 0 & -c
\end{matrix} \Big ): c\in \mathbb{C} \Big \}$, $\mathfrak p^+ = \Big \{ \Big (\begin{matrix} 0 & c\\ 0 & 0 \end{matrix} \Big ):c\in \mathbb{C} \Big \}$,
$\mathfrak p^- = \Big \{ \Big (\begin{matrix} 0 & 0\\ c & 0
\end{matrix} \Big ):c\in \mathbb{C} \Big \}$ are spanned by $h$, $x$ and $y$, respectively.
The product $\mathbb K^\mathbb C P^- = \Big \{ \left(
                             \begin{array}{cc}
                               a & 0 \\
                               b & \frac{1}{a} \\
                             \end{array}
                           \right)
 : 0\not = a\in \mathbb{C}, b \in \mathbb{C} \Big \}$ is
a closed subgroup to be denoted $T$; its Lie algebra is $\mathfrak t
= \mathbb C h + \mathbb C y$.  The product set $P^+\mathbb K^\mathbb
CP^-=P^+T$
is dense open in $G^\mathbb C$, contains $G$, and the product
decomposition of each of its elements is unique. ($G^{\mathbb C}/T$
is the Riemann sphere, $g \tilde{\mathbb K} \to g T,\:(g\in G)$ is
the natural embedding of $\mathbb D$ into it.)

According to holomorphic induction \cite[Chap 13]{Kir} 
the isomorphism classes of homogeneous holomorphic  vector bundles are
in one to one correspondence with equivalence classes of linear
representations $\varrho$ of the pair $(\mathfrak t, \tilde{\mathbb
K})$.  Since $\tilde{\mathbb K}$ is connected, here this means just
the representations of $\mathfrak t$.  Such a representation is
completely determined by the two linear transformations $\varrho(h)$
and $\varrho(y)$ which satisfy the bracket relation of $h$ and $y$,
that is,
\begin{equation} \label{triangalgcommrel}
[\varrho(h),\varrho(y)] = - \varrho(y).
\end{equation}
The $\tilde{G}$-invariant Hermitian structures on the homogeneous
holomorphic vector bundle (making it into a homogeneous holomorphic
Hermitian vector bundle), if they exist, are given by
$\varrho(\tilde{\mathbb K})$-invariant inner products on the
representation space.  An inner product is $\varrho(\tilde{\mathbb
K})$-invariant if and only if $\varrho(h)$ is diagonal with real
diagonal elements in an appropriate basis.

We will be interested only in bundles with a Hermitian structure.
So, we will assume without restricting generality, that the
representation space of $\varrho$ is $\mathbb C^d$ and that
$\varrho(h)$ is a real diagonal matrix.

Furthermore, we will be interested only in irreducible homogeneous
holomorphic Hermitian vector bundles, this corresponds to $\varrho$
not being the orthogonal direct sum of non-trivial representations.
Suppose we have such a $\varrho$; we write $V_\alpha$ for the
eigenspace of $\varrho(h)$ with eigenvalue $\alpha$.  Let $-\eta$ be
the largest eigenvalue of $\varrho(h)$ and $m$ be the largest
integer such that $-\eta, -(\eta+1), \ldots , -(\eta + m)$ are all
eigenvalues.  From \eqref{triangalgcommrel} we have $\varrho(y)
V_\alpha \subseteq V_{\alpha -1}$; this and orthogonality of the
eigenspaces imply that $V=\oplus_{j=0}^mV_{-(\eta +j)}$ and its
orthocomplement are invariant under $\varrho$.  So, $V$ is the whole
space, and have proved that the eigenvalues of $\varrho(h)$ are
$-\eta, \ldots , -(\eta+m)$.

>From this it is clear that $\varrho$ can be written as  the tensor
product of the one dimensional representation $\sigma$ given by
$\sigma(h) = -\eta$, $\sigma(y)=0$, and the representation
$\varrho^0$ given by $\varrho^0(h)= \varrho(h) + \eta I$,
$\varrho^0(y) = \varrho(y)$. Correspondingly, the bundle for
$\varrho$ is the tensor product of a line bundle $L_\eta$ and the
bundle corresponding to $\varrho^0$.

The representation $\varrho^0$ has the great advantage that it
lifts to a holomorphic representation of the group $T$.  It follows
that the homogeneous holomorphic vector bundle it determines for
$\mathbb D, \tilde{G}$, can be obtained as the restriction to
$\mathbb D$ of the homogeneous holomorphic vector bundle over
$G^\mathbb C/T$ obtained by ordinary induction in the complex
analytic category.  So, (as a convenient choice) take the local
holomorphic cross section $z\mapsto s(z):= \Big (\begin{matrix}
1&z\\ 0&1 \end{matrix}\Big )$ of $G^\mathbb C/T$ over $\mathbb D$.
In the trivialization given by $s(z)$, the multiplier then appears
for $g=\Big (\begin{matrix} a & b\\ c & d \end{matrix}\Big ) \in
G^\mathbb C$ as
\begin{eqnarray} \label{explicitformulamult}
J_g^0(z) &=& \varrho^0\big ( s(z)^{-1} g^{-1} s(g \cdot z) \big ) \nonumber \\
&=& \varrho^0\begin{pmatrix} c z + d & 0\\ - c & (c z + d )^{-1} \end{pmatrix}   \nonumber\\
&=& \varrho^0\Big (\exp \big (\frac{-c}{ c z + d}y  \big ) \Big )
\varrho^0 \big ( \exp(2 \log(c z + d) h ) \big ).
\end{eqnarray}
The last two equalities are simple computations.

For the line bundle $L_\eta$, the multiplier is $g^\prime(z)^\eta$
(we write $g^\prime(z) = \frac{\partial g}{\partial z}(z)$).
Consequently, the multiplier corresponding to the original $\varrho$
is
\begin{equation} \label{genmult}
 J_g(z) = \big (g^\prime(z)\big )^\eta J_g^0(z).
\end{equation}

\section{Conditions imposed by the reproducing kernel}
We now assume that we have a homogeneous holomorphic vector bundle
induced by $\varrho$ as in the preceding sections and that it has a
reproducing kernel. Then we derive conditions about the action of
$\tilde{G}$ that follow from this hypothesis.  In the final section,
we will show that these conditions are sufficient: they lead
directly to the construction of all homogeneous operators the
Cowen-Douglas class\index{Cowen-Douglas class} with
multiplicity free\index{multiplicity free} representations.

Under our hypothesis there is a Hilbert space structure on our
sections in which the action of $\tilde{G}$ given by
\eqref{eq:homker def} is unitary.  We will study this representation
through its $\mathbb K$ - types (i.e., its restriction to
$\tilde{\mathbb K}$).  We first compute the infinitesimal
representation.

For $X \in \mathfrak g$, and holomorphic $f$, we have
\begin{eqnarray} \label{Udirect}
\lefteqn{ (U_Xf)(z) := \big (\tfrac{d}{dt} \big )_{|t=0}  \big( U_{\exp(tX)}
f\big )(z)} \nonumber \\
&=& \big (\tfrac{d}{dt} \big )_{|t=0} \Big \{ \Big ( \frac{\partial
\big ( \exp (-tX) \cdot z \big )}{\partial z}\Big )^\eta J^0_{\exp
(-tX)}(z) f(\exp(-tX)\cdot z) \Big \}.
\end{eqnarray}

There is a local action of $G^\mathbb C$, so this formula remains
meaningful also for $X\in \mathfrak g^\mathbb C$. There are three
factors to differentiate. For the last one, $\big (\tfrac{d}{dt}
\big )_{|t=0} f(\exp(-tX)\cdot z) = - (X z) f^\prime(z)$, and we see
that $\exp(t x) \cdot z =  \left(
                              \begin{array}{cc}
                                1 & t \\
                                0 & 1 \\
                              \end{array}
                            \right)
 \cdot z = z+t$ gives $x
\cdot z = 1$; by similar computations, $y \cdot z= -z^2$, $h\cdot z=
z$.  For the first factor, we interchange the differentiations and
get $- \eta \frac{\partial}{\partial z} (X\cdot z)$, i.e., $0, 2\eta
z, -\eta$,   respectively for $x,y$ and $h$.


To differentiate the factor in the middle, we use its expression
\eqref{explicitformulamult}.  First for $X=y$, we have

\begin{align}
\left . \frac{d}{dt} \right |_{t=0}\varrho^0\big ( \exp (-t
(tz+1)^{-1}y) \big ) & = \left . \frac{d}{dt} \right |_{t=0}\big (
\exp (-t (tz+1)^{-1}
\varrho^0(y) \big )\nonumber \\
& = - \varrho^0(y)
\end{align}
and
\begin{align}
\left . \frac{d}{dt} \right |_{t=0}\varrho^0(\exp (2 \log(tz +1) h))
& = \left . \frac{d}{dt} \right |_{t=0}\exp (2 \log(tz +1) \varrho^0(h))\nonumber\\
& = 2 z \varrho^0(h)
\end{align}
>From these, following the conventions of \cite{Sig} in defining H,E,F, it follows
that
\begin{align} \label{-Ff}
(F f)(z) :=  ( U_{-y} f ) (z) & = \left .\frac{d}{dt}\right |_{|t=0}
J_{\exp (t y)}(z) f(\exp(ty)
\cdot z) \nonumber\\
& = \big ( -2 \eta z I + 2 z \varrho^0(h) - \varrho^0(y) \big )f(z)
-z^2 f^\prime(z).
\end{align}
Similar, simpler computations give, for $g = \exp (t x) =
\left(
  \begin{array}{cc}
    1 & t \\
    0 & 1 \\
  \end{array}
\right)$
\begin{equation}
(Ef)(z) := \big (U_x f\big )(z) = - f^\prime(z).
\end{equation}
Finally, for $g= \exp (th) = \left(
                               \begin{array}{cc}
                                 e^{t/2} & 0 \\
                                 0 & e^{-t/2} \\
                               \end{array}
                             \right)$, we have
$$ J_{\exp (t h)}(z) = \varrho \left(
                                 \begin{array}{cc}
                                   e^{-t/2} & 0 \\
                                   0 & e^{t/2} \\
                                 \end{array}
                               \right)
= \exp(-t) \varrho^0(h).
$$
Hence it is not hard to verify that
\begin{equation}\label{missing}
(H f)(z):= \big ( U_hf \big ) (z) =  \big ( - \eta I + \varrho^0(h)
\big ) f(z) - z f^\prime(z).
\end{equation}
Under our hypothesis, we have a reproducing kernel and $U$ is
unitary. From our computations above, we can determine how $U$
decomposes into irreducibles. The infinitesimal representation of
$U$ acts on the vector valued polynomials; a good basis for this
space is $\{\varepsilon_j z^n: n \geq 0\}$; $\varepsilon_j$ is the
$j$th natural basis vector in $\mathbb C^d$.  We have
$H(\varepsilon_j z^n) = - (\eta + j + n)(\varepsilon_j z^n)$, so the
lowest $\mathbb K$ - types of the irreducible summands are spanned
by the $\varepsilon_j$.  This space is also the kernel of $E$.  So,
$U$ is direct sum of discrete series representations ($U^{\eta +
j}$, in the notation of \cite{Sig}), each one
appearing as many
times as $-(\eta + j)$ appears on the diagonal of $\varrho(h)$.

\section{The multiplicity-free case}

In order to be able to use the computations of
\cite{KM} without
confusion, we introduce the parameter $\lambda = \eta +\frac{m}{2}$.

From the last remark of the preceding section, it is clear that if
$U$ is multiplicity-free\index{multiplicity free} then
$\varrho(h)$ is an $(m+1)\times (m+1)$
matrix with eigenvalues $-\lambda+\frac{m}{2}, -\lambda +
\frac{m}{2}-1, \ldots, -\lambda-\frac{m}{2}$.  As $\varrho(h)
\varepsilon_j = -(\lambda-\frac{m}{2} + j) \,\varepsilon_j$,
\eqref{triangalgcommrel} shows that
$$ \varrho(h) \big ( \varrho(y) \varepsilon_j\big ) = - (\lambda+\frac{m}{2}+j+1) \,\varrho(y)\, \varepsilon_j, \mbox{\rm  that is}, \varrho(y)\,\varepsilon_j = \mbox{\rm const}\, \varepsilon_{j+1}.$$
So, $\varrho(y)$ is a lower triangular matrix (with non-zero
entries, otherwise we have a reducible bundle).  The homogeneous
holomorphic vector bundle determines $\varrho(y)$ only up to a
conjugacy by a matrix commuting with $\varrho(h)$, that is, a
diagonal matrix.  So, we can choose the  realization of our bundle
by applying an appropriate conjugation such that $\varrho(y)= S_m$,
the triangular matrix whose $(j,j-1)$ element is $j$ for $1   \leq j
\leq m$.

By standard representation theory of $\mathrm{SL}(2, \mathbb R)$,
the vectors $(-F)^n \varepsilon_j$ are orthogonal and the
irreducible subspaces $\mathcal H^{(j)}$ for $U$ are
$\mathrm{span}\{(-F)^n \varepsilon_j: n \geq 0\}$ for $0\leq j \leq
m$. There is also precise information about the norms.


Using this, we can construct an orthonormal basis  for our
representation space.

For any $n \geq 0$, we let $u^j_n(z) = (-F)^n \varepsilon_j$.

To proceed further, we need to find the vectors $u^j_n(z)$
explicitly. This is facilitated by the following Lemma.

\begin{Lemma} \label{-F}
Let $u$ be a vector with $u_\ell(z) = u_\ell z^{n-\ell}$, $0\leq
\ell \leq m$ and $n \geq 0$.  We then have
$$
(-Fu)_\ell(z) = (2\lambda-m+\ell+n)u_\ell z^{n+1-\ell} + \ell
u_{\ell-1}z^{n+1-\ell},\, 0\leq \ell \leq m.
$$
\end{Lemma}
\begin{proof}
We recall (\ref{-Ff}) that $-(Ff)(z) = 2\lambda z f(z) + S_m f(z) -
2 z D_m f(z) + z^2 f^\prime(z)$ for $f \in {\mathcal {H}}(n)$, where
$D_m=-\varrho^0(h)$ is the diagonal operator with diagonal
$\{-\frac{m}{2},-\frac{m}{2}+1, \ldots, \frac{m}{2}\}$
and $S_m$ is the forward weighted shift with weights $1,2,\ldots,
m$. Therefore we have
$$
(-Fu)_\ell(z) = \big (2\lambda u_\ell+ \ell u_{\ell-1} -
(m-2\ell)u_\ell + (n-\ell) u_\ell \big ) z^{n+1-\ell}$$ completing
the proof.
\end{proof}

\begin{Lemma} \label{u}
For $0 \leq j \leq m$ and $0 \leq \ell \leq m$, we have
$$
u^j_{n, \ell}(z) =
\begin{cases} 0 & \mbox{\rm if ~} {0 \leq \ell \leq j-1} \\
\binom{n}{k}(j+1)_k(2\lambda -m + 2j + k)_{n-k} z^{n-k}& \mbox{\rm
if ~} {j \leq \ell \leq m,~ k=\ell -j,}
\end{cases}
$$
where $u^j_{n,\ell}(z)$ is the scalar valued function at the
position $\ell$ of the $\mathbb{C}^{m+1}$ - valued function $u^j_n(z):=(-F)^n
\varepsilon_j$.
\end{Lemma}

\begin{proof} The proof is by induction on $n$.
The vectors $u^j_n$ are in
${\mathcal {H}}(n)$ for $0\leq j \leq m$.  For a fixed but arbitrary
positive integer $j$, $0\leq j \leq m$, we see that
$u^j_{n,\ell}(z)$ is $0$ if $n < \ell -j$. We have to verify that
$(-F u^j_n)(z) = u^j_{n+1}(z)$.  From the previous Lemma, we have
$$(-F u^j_n)_\ell(z) =
(2\lambda-m+\ell+n+j) u^j_{n,\ell} z^{n+j+1-\ell} + \ell
u^j_{n,\ell-1}z^{n+j+1-\ell},$$ where $(-F u^j_n)_\ell(z)$ is the
scalar function at the position $\ell$ of the $\mathbb C^{m+1}$ -
valued function $(-F u^j_n)(z)$. To complete the proof, we note
(using $k=\ell-j$) that
\begin{eqnarray*}
\lefteqn{ (-F u^j_n)_{j+k}(z)}\\
&=& \big (\tbinom{n}{k}(j+1)_k(2\lambda -m +2j+k)_{n-k}
(2\lambda-m+2j+k+n)+ \\
&& \tbinom{n}{k -1}(j+1)_k(2\lambda -m +2j+k
-1)_{n-k} \big ) z^{n+1-k}\\
&=& (j+1)_k(2\lambda-m+2j+k)_{n-k}\\
&& \big ( \tbinom{n}{k}(2\lambda -m + 2j + k + n)
+ \tbinom{n}{k-1} (2\lambda -m + 2j +k -1) \big )z^{n+1-k}\\
&=& (j+1)_k(2\lambda-m+2j+k)_{n-k} \\ && \big( (\tbinom{n}{k} +
\tbinom{n}{k-1}(2\lambda-m+2j+k -1)+
(n+1)\tbinom{n}{k}\big )z^{n+1-k}\\
&=& (j+1)_k(2\lambda-m+2j+k)_{n-k} \\&&
\big(\tbinom{n+1}{k}(2\lambda-m+2j+k
-1) +\tbinom{n+1}{k} (n-k+1) \big )z^{n+1-k}\\
&=& (j+1)_k(2\lambda-m+2j+k)_{n-k} \big(\tbinom{n+1}{k}(2\lambda -m
+
2j +n) \big )z^{n+1-k}\\
&=&(j+1)_k \big(\tbinom{n+1}{k}(2\lambda-m+2j+k)_{n+1-k} \big )z^{n+1-k}\\
&=& u^j_{n+1,j+k}(z)
\end{eqnarray*}
for a fixed but arbitrary $j$, $0 \leq j \leq m$ and $k$, $ ~0\leq k
\leq m-j$. This completes the proof.
\end{proof}

On $\mathcal H^{(j)}$, we have the representation $U^{\lambda_j}$
acting $(0 \leq j \leq m)$, where $\lambda_j=\lambda -\frac{m}{2}
+j$.  Its lowest $\mathbb K$ - type is spanned by $\varepsilon_j \,
(= u_0^j)$  and $H\varepsilon_j = \lambda_j \varepsilon_j$.  By
\cite[Prop 6.14]{Sig} we have
$\|(-F)^k\varepsilon_j \|^2 = \sigma_k^j \|(-F)^{k-1}\varepsilon_j\|^2$ with
$$
\sigma_k^j = (2 \lambda_j + k -1) k
$$
for all $k \geq 1$.  (Here we used that the constant $q$ in
\cite[equation (6.33)]{Sig} equals
$\lambda_j(1-\lambda_j)$ by
\cite[Theorem 6.2]{Sig}.)  We write
$$
\boldsymbol{\sigma}^j_n = \prod_{k=1}^{n} \sigma^j_k
$$
which can be written in a compact form
\begin{equation}\label{nsigma}
\boldsymbol{\sigma}_n^j = ( (2\lambda_j)_n(1)_n),
\end{equation}
where $(x)_n = (x+1)\cdots (x+n-1)$.  We stipulate that the binomial
co-efficient $\binom{n}{k}$ as well as $(x)_{n-k}$ are both zero if
$n < k$.

The positivity of the normalizing constants $\big (
\boldsymbol{\sigma}_{n-j}^j\big )^{\frac{1}{2}}$ ($n \geq j$) is
equivalent to the existence of an inner product for which the set of
vectors $\mathbf{e}^j_{n-j}$ defined by the formula:
$$\mathbf{e}^j_{n-j} = (\boldsymbol{\sigma}_n^j)^{-\frac{1}{2}}u^j_{n-j}(z), ~n\geq j, ~0\leq j\leq m$$  forms an orthonormal set.   Of course, the positivity condition is fulfilled  if and only if $2\lambda > m$.

In this way, for fixed $j$, each $e^j_{n-j}$ has the same norm for
all $n\geq j$.  Hence the only possible choice for an orthonormal
system is $\{\mu_j e_{n-j}^j: n \geq j\}$ for some positive real
numbers $\mu_j > 0$ ($0 \leq j \leq m$). However, we may choose the
norm of the first vector, that is, the vector $\mathbf{e}^j_0$,
$0\leq j \leq m$, arbitrarily. Therefore, all the possible choices
for an orthonormal set are
\begin{eqnarray} \label{onbj}
\mu_j \mathbf{e}^j_{n-j}(z) &=& \frac{\mu_j}{\sqrt{(2\lambda - m +
2j)_{n-j}}\sqrt{(1)_{n-j}}} u^j_{n-j}(z),
\end{eqnarray}
$n\geq j, ~0\leq j\leq m,$ and $\mu_j,~0\leq j \leq m$ are $m+1$ arbitrary positive numbers.

Let us fix a positive real number $\lambda$ and $m\in \mathbb{N}$
satisfying $2 \lambda > m$. Let $\mathcal{H}^{(\lambda,{\boldsymbol
\mu})}$ denote the closed linear span of the vectors $\{\mu_j
e^j_{n-j}:~ 0\leq j \leq m, ~ n \geq j\}$. Then the Hilbert space
$\mathcal{H}^{(\lambda,{\boldsymbol \mu})}$ is the representation space
for $U$ defined in (\ref{eq:grp action}). Since the vectors
$u^j_n\perp u^k_p$ as long as $j\not= k$, it follows that the
Hilbert space $\mathcal{H}^{(\lambda,{\boldsymbol \mu})}$ is the
orthogonal direct sum $\oplus_{j=0}^m\, \frac{1}{\mu_j}
\mathcal{H}^{(j)}$.

We proceed to compute the reproducing kernel by using the
orthonormal system $\{\mu_j e_{n-j}^j: n \geq j\}$, $0 \leq j \leq
m$.
We point out that for $0\leq \ell \leq m$, the entry
$\mathbf{e}^{\ell,j}_{n-j} z^{n-j}$ at the position $\ell$ of the
vector $\mathbf{e}^j_{n-j}(z)$ is $0$ for $n < \ell$.
Consequently, $\mathbf{e}^j_{n-j}$ is the zero vector unless $n
\geq j$. The set of vectors $\{\mu_j \mathbf{e}^j_{n-j}: 0\leq j
\leq m,~ n \geq j\}$ is orthonormal in the Hilbert space
$\mathcal{H}^{(\lambda,{\boldsymbol \mu})}$.  We note that
$$\mathbf{e}^{j}_{n-j}(z) = (\!( e^{\ell, j}_{n-j} z^{n-k}
)\!)_{\ell=0}^m\,,$$
\begin{eqnarray} \label{e}
\lefteqn{\big (\mathbf{e}^{j}_{n-j}(z) \big )_\ell}\nonumber \\
&=& \begin{cases} 0, & 0 \leq \ell \leq j-1 \\
\sqrt{\frac{(2\lambda + 2j -m + k)_{n-j-k}}{(1)_{n-j-k}}}
\sqrt{\frac{(n-j-k+1)_k}{(2\lambda + 2j -m)_k}}
\frac{(j+1)_k}{(1)_k} z^{n-k}, & j \leq \ell \leq m,~
k=\ell -j.
\end{cases}
\end{eqnarray}

We have under the hypothesis that we have a reproducing kernel
Hilbert space on which the representation $U$ is unitary, explicitly
determined an orthonormal basis for this space.  Now we are able to
answer the question of whether this space really exists.  For this
it is enough to show that $\sum e_n(z) \, \overline{e_n(w)}^{\rm
tr}$ converges pointwise, the sum then represents the reproducing
kernel for this Hilbert space.  We will sum the series explicitly,
and will verify that it gives exactly the kernels constructed in
\cite{KM}.
This will complete the program of this paper by proving that the
examples of \cite{KM}  
give all the homogeneous operators in
the Cowen-Douglas class\index{Cowen-Douglas class} whose
associated representation is multiplicity free.\index{multiplicity free}


To compute the kernel function, it is convenient to set, for any $n
\geq 0$,
\begin{eqnarray}\label{G}
G({\boldsymbol \mu},n,z) &=& \begin{pmatrix}
\mu_0 e^{0,0}_n z^n & \hdots & 0 & \hdots & \ 0 \\
\vdots & \hdots & \vdots & \hdots & \vdots\\
\mu_0 e^{j,0}_{n}z^{n-j}&\hdots&\mu_j e^{j,j}_{n-j}z^{n-j}&\hdots &
0\\
\vdots & \hdots & \vdots & \hdots & \vdots \\
\mu_0e^{m,0}_n z^{n-m}&\hdots&\mu_j e^{m,j}_{n-j} z^{n-m}&\hdots&
\mu_m e^{m,m}_{n-m} z^{n-m}
\end{pmatrix} \nonumber\\
&=& \begin{pmatrix}
z^n &\hdots & 0\\
\vdots &\ddots&\vdots \\
0&\hdots & z^{n-m}
\end{pmatrix}
\begin{pmatrix}
e^{0,0}_n & \hdots & 0 & \hdots & \ 0 \\
\vdots & \hdots & \vdots & \hdots & \vdots\\
e^{j,0}_{n} &\hdots& e^{j,j}_{n-j} &\hdots &
0\\
\vdots & \hdots & \vdots & \hdots & \vdots \\
e^{m,0}_n & \hdots & e^{m,j}_{n-j} &\hdots& e^{m,m}_{n-m}
\end{pmatrix}
\begin{pmatrix}
\mu_0 &\hdots & 0\\
\vdots &\ddots&\vdots \\
0&\hdots & \mu_m
\end{pmatrix} \nonumber \\
&=& D_n(z) G(n) D({\boldsymbol \mu}),
\end{eqnarray}
where $D_n(z),~ D({\boldsymbol \mu})$ are the two diagonal matrices
and $G(n) = (\!\! ( e^{\ell, j}_{n-j} )\!\!)_{\ell, j = 0}^m$ with
$\mathbf{e}^{\ell, j}_{n-j} = 0$ if $\ell <j$ or  if $n < \ell$.
 The nonzero entries of the lower triangular matrix $G(n)$, using (\ref{e}), are
\begin{eqnarray} \label{G(j+k,j)}
\lefteqn{G_{j+k,j}(n) =  \frac{\tbinom{n-j}{k} (j+1)_k (2\lambda-m + 2j
+k)_{n-j-k}} {\sqrt{(2\lambda-m+2j)_{n-j}} \sqrt{(1)_{n-j}}}} \nonumber\\
&=& \frac{\sqrt{(2\lambda - m +2j
+k)_{n-j-k}}}{\sqrt{(2\lambda-m+2j)_k}}
\frac{(n-j-k+1)_k}{\sqrt{(1)_{n-j}}}
\frac{(j+1)_k}{(1)_k}\nonumber \\
&=& \sqrt{\frac{(2\lambda - m +2j +k)_{n-j-k}}{(2\lambda-m+2j)_k}}
\sqrt{\frac{(n-j-k+1)_k}{(1)_{n-j-k}}}\frac{(j+1)_k}{(1)_k}
\end{eqnarray}
for $0 \leq k \leq m-j$.

Now, we are ready to compute the reproducing kernel $K_j$ for the
Hilbert space $\mathcal H^{(j)} = {\rm span}\{e^j_{n-j}:n \geq j\}$,
$0\leq j \leq m$. Recall that $K(z,w) = \sum_{n=0}^\infty e_n(z)
e_n(w)^*$ for any orthonormal basis $e_n,~n \geq 0$. This ensures
that $K$ is a {\em positive definite} kernel. For our computations,
we will use the particular orthonormal basis $\mathbf{e}^j_{n-j}$
as described in (\ref{onbj}). Since there are $j$ zeros at the top
of each of these basis vectors, it follows that $(\ell,p)$ will be
$0$ if either $\ell < j$ or $p < j$. We will compute $(\!\!(
K_j(z,w) )\!\!)$, at $(\ell,p)$ for $ j \leq \ell , p \leq m$. For
$\ell, p$ as above, we have
\begin{eqnarray*}
(\!\!( K_j(z,w) )\!\!)_{\ell,p} &=& \sum_{n\geq \mbox{\rm
max}(\ell,p)}^\infty \mathbf{e}^j_{n-j, \ell}(z)
\overline{\mathbf{e}^j_{n-j,p}(w)} \\
&=& \sum_{n\geq \mbox{\rm max}(\ell, p)}^\infty G_{\ell,j}(n)
G_{p,j}(n) z^{n-\ell}\bar{w}^{n-p}.
\end{eqnarray*}
We first simplify the co-efficient $G_{\ell,j}(n) G_{p,j}(n)$ of
$z^{n-\ell}\bar{w}^{n-p}$. The values of $G_{\ell,j}(n)$ are given
in (\ref{G(j+k,j)}). Therefore, we have
\begin{eqnarray*}
\lefteqn {G_{\ell,j}(n) G_{p,j}(n)}\\
&=& \Big (\frac{(2\lambda_j + \ell
-j)_{n-\ell}}{(2\lambda_j)_{\ell-j}}
\frac{(n-\ell+1)_{\ell-j}}{(1)_{n-\ell}} \frac{(2\lambda_j + p
-j)_{n-p}}{(2\lambda_j)_{p-j}} \frac{(n-\ell+1)_{\ell-j}}{(1)_{n-p}}
\Big )^{1/2}\\
&& \times \frac{(j+1)_{\ell-j}}{(1)_{\ell-j}}
\frac{(j+1)_{p -j}}{(1)_{p -j}}\\
&=& \frac{(2\lambda_j + p -j)_{n-p} (n-\ell+1)_{\ell-j}}
{(2\lambda_j)_{\ell-j}(1)_{n-p}} \Big ( \frac{(2\lambda_j+\ell
-j)_{p-\ell} (n-p+1)_{p-\ell}} {(2\lambda_j+\ell -j)_{p-\ell}
(n-p+1)_{p-\ell}} \Big )^{1/2}\\
&& \times \frac{(j+1)_{\ell-j}}{(1)_{\ell-j}}
\frac{(j+1)_{p -j}}{(1)_{p -j}}\\
&=& \frac{(2\lambda_j)_{p-j}(2\lambda_j + p -j)_{n-p}
(n-\ell+1)_{\ell-j} (n-p+1)_{p-j}} {(2\lambda_j)_{p-j}
(2\lambda_j)_{\ell-j}(1)_{n-p}(n-p+1)_{p-j}}
\frac{(j+1)_{\ell-j}}{(1)_{\ell-j}} \frac{(j+1)_{p -j}}{(1)_{p -j}}\\
&=& \frac{(2\lambda_j)_{n - j} (n-\ell+1)_{\ell-j}
(n-p+1)_{p-j}}{(2\lambda_j)_{p-j} (2\lambda_j)_{\ell-j}(1)_{n-j}}
\frac{(j+1)_{\ell-j}}{(1)_{\ell-j}} \frac{(j+1)_{p -j}}{(1)_{p -j}}.
\end{eqnarray*}

\begin{theorem} \label{kernel^lam_mu}
Given an arbitrary set $\mu_0, \ldots , \mu_m$ of positive numbers,
and $2\lambda > m$, we have
$$K^{(\lambda,{\boldsymbol \mu})}(z,w) = \sum_{j=0}^m \mu_j^2 K_j(z,w) =
\mathbf{B}^{(\lambda,{\boldsymbol \mu})}(z,w).$$ As a result, the
two Hilbert spaces ${\mathcal {H}}^{(\lambda,{\boldsymbol \mu})}$ and
$\mathbf{A}^{(\lambda,{\boldsymbol \mu})}$ of
\cite{KM} are equal.
\end{theorem}

\begin{proof}
We  now compare the co-efficients  $(\!\!( K_j(z,w) )\!\!)_{\ell,p}$
with that of a known Kernel.  Let $B^{\lambda_j}(z,w) =
(1-z\bar{w})^{-2\lambda_j}$, where $B(z,w) = (1-z\bar{w})^{-2}$ is
the Bergman kernel on the unit disc.  We let $\partial$ and
$\bar{\partial}$ denote differentiation with respect to $z$ and
$\bar{w}$ respectively. Put
$$
\tilde{\mathbf{B}}^{(\lambda_j)}(z,w) = (\!\!(
\partial^{\ell-j}\bar{\partial}^{p-j} (1-z\bar{w} )^{-2\lambda_j}
)\!\!)_{j \leq \ell, p \leq m}.
$$
We expand the entry at the position $(\ell, p)$ of
$\tilde{\mathbf{B}}^{(\lambda_j)}(z,w)$ to see that
\begin{eqnarray*}
\lefteqn{(\!\!(\tilde{\mathbf{B}}^{(\lambda_j)}(z,w) )\!\!)_{\ell,p}}\\
 &=& \sum_{\nu \geq \mbox{\rm max}(\ell-j,p-j)}
\frac{(2\lambda_j)_{\nu}}{(1)_{\nu}} (\nu-\ell+j+1)_{\ell-j} (\nu +
j - p +1)_{p-j}
z^{\nu-(\ell-j)}\bar{w}^{\nu-(p-j)}\\
&=& \sum_{n \geq \mbox{\rm max}(\ell,p)}
\frac{(2\lambda_j)_{n-j}}{(1)_{n-j}} (n -\ell +1)_{\ell-j} (n - p
+1)_{p-j} z^{ n - \ell }\bar{w}^{n - p },
\end{eqnarray*}
where we have set $n=m+j$.  Comparing these coefficients with that
of $G_{\ell,j}(n) G_{p,j}(n)$, we find that
\begin{equation}\label{kernel^beta_j}
K_j(z,w) = D_j \tilde{\mathbf{B}}^{(\lambda_j)}(z,w) D_j,
\end{equation}
\mbox{~where~} $D_j$ is a diagonal matrix with
$\frac{1}{(2\lambda_j)_{\ell-j}}
\frac{(j+1)_{\ell-j}}{(1)_{\ell-j}}$ at the $(\ell, \ell)$ position
with $j \leq \ell \leq m$. Hence $K_j(z,w) =
\mathbf{B}^{(\lambda_j)}(z,w)$ which was defined in the equation
(\cite[equation (4.3)]{KM}).

Clearly, we can add up the kernels $K_j$ to obtain the kernel
$K^{(\lambda, {\boldsymbol \mu})}$ for the Hilbert space
$\mathcal{H}^{(\lambda, {\boldsymbol \mu})} = \oplus_{j=0}^m
\frac{1}{\mu_j} {\mathcal {H}}^{(j)}$.  Hence the proof of the theorem
is complete.
\end{proof}

\begin{Corollary}
The irreducible homogeneous operators in the Cowen - Douglas class
whose associated representation is multiplicity free\index{multiplicity free}
are exactly the
adjoints of $M^{(\lambda, {\boldsymbol \mu})}$ constructed in
\cite{KM}.
\end{Corollary}
\begin{proof}
In our discussion up to here we proved that the Hilbert space
$\mathcal H^{(\lambda, {\boldsymbol \mu})}$ corresponding to a
homogeneous operator in the Cowen - Douglas class\index{Cowen-Douglas class} has a reproducing
kernel given by $K^{(\lambda, {\boldsymbol \mu})}=\sum_0^m \mu_j^2
K_j$, $2\lambda
> 1,\: \mu_1, \ldots , \mu_m >0$. It follows from the Theorem that
the kernels obtained this way are the same as (are equivalent to)
the kernels constructed in \cite{KM}.
These operators were shown to
be irreducible \cite{KM}.
\end{proof}

We now consider the action of the multiplication operator
$M^{(\lambda,{\boldsymbol \mu})}$ on the Hilbert space
$\mathcal{H}^{(\lambda,{\boldsymbol \mu})}$. Let $\mathcal{H}(n)$ be the
linear span of the vectors
$$\{\mathbf{e}^0_{n}(z), \ldots, \mathbf{e}^j_{n-j}(z), \ldots , \mathbf{e}^m_{n-m}(z)\},$$
where as before, for $0\leq \ell \leq m$,
$\mathbf{e}^j_{n-\ell}(z)$ is zero if $n-\ell < 0$. Clearly,
${\mathcal H}^{(\lambda,{\boldsymbol \mu})} = \oplus_{n=0}^\infty {\mathcal
H}(n)$. We have
\begin{eqnarray*}
z G(n, z) &=&  D_n(z) G(n) D({\boldsymbol \mu})\\
&=&  D_{n+1}(z) G(n) D({\boldsymbol \mu})\\
&=& D_{n+1}(z) G(n+1)D({\boldsymbol \mu}) \big (D({\boldsymbol
\mu})^{-1} G(n+1)^{-1} G(n) D({\boldsymbol \mu}) \big ).
\end{eqnarray*}
If we let $W(n) = D({\boldsymbol \mu})^{-1} G(n+1)^{-1} G(n)
D({\boldsymbol \mu})$, then we see that $z {\mathbf e}_{n-j}^j(z)
= G({\boldsymbol \mu}, n+1, z) W_j(n)$, where $W_j(n)$ is the $j$th
column of the matrix $W(n)$. It follows that the operator
$M^{(\lambda,{\boldsymbol \mu})}$ defines a block shift $W$ on the
representation space $\mathcal{H}^{(\lambda,{\boldsymbol \mu})}$. The
block shift $W$ is defined by the requirement that $W:\mathcal H(n)
\to \mathcal H(n+1)$ and $W_{|\mathcal{H}(n)} = W_n^{\mathrm tr}$.

Here, we have a construction of the representation space $\mathcal
H^{(\lambda, {\boldsymbol \mu})}$ along with the matrix
representation of the operator $M^{(\lambda,{\boldsymbol \mu})}$
which is independent of the corresponding results from
\cite{KM}. 

\section{Examples}

Recall that $G({\boldsymbol \mu}, n,z) =  D_n(z) G(n) D({\boldsymbol
\mu})$. Once we determine the matrix $G(n)$ explicitly, we can
calculate both the block weighted shift and the kernel function.

We discuss these calculations in the particular case of  $m=1$.
First, it is easily seen that
\begin{equation} \label{G2}
G(n) = \begin{pmatrix}
\big ( \frac{(2\lambda - 1)_n}{(1)_n} \big )^{1/2} & 0 \\
(\frac{n}{2\lambda - 1})^{1/2} \big (
\frac{(2\lambda)_{n-1}}{(1)_{n-1}} \big )^{1/2} &
\big ( \frac{(2\lambda +1)_{n-1}}{(1)_{n-1}} \big )^{1/2}\\
\end{pmatrix}.
\end{equation}
The block  $W_n$ of the weighted shift $W$ is
\begin{equation} \label{W2}
W_n =  \begin{pmatrix}
(\frac{n+1}{2\lambda + n - 1})^{1/2} & 0\\
-\frac{1}{\mu_1} (\frac{2 \lambda}{2\lambda -1})^{1/2} (\frac{1}
{(2\lambda+n-1)(2\lambda +n)})^{1/2} & (\frac{n}{2 \lambda+n})^{1/2}
\end{pmatrix}.
\end{equation}

Finally, the reproducing kernel $K^{(\lambda, {\boldsymbol \mu})}$
with $m=1$ is easily calculated:
\begin{equation} \label{K2}
K^{(\lambda,{\boldsymbol \mu})}(z,w) = \begin{pmatrix}
\frac{1}{(1-\bar{w}z)^{2\lambda-1}}& \frac{z}{(1-\bar{w}z)^{2 \lambda}}\\
\frac{\bar{w}}{(1-\bar{w}z)^{2 \lambda}} & \frac{1}{2\lambda -1}
\frac{1+(2\lambda-1)\bar{w}z}{(1-\bar{w}z)^{2\lambda+1}}
\end{pmatrix} + \mu_1^2 \begin{pmatrix}
0&0\\
0&\frac{1}{(1-\bar{w}z)^{2\lambda +1}}
\end{pmatrix}.
\end{equation}

One might continue the explicit calculations, as above, in the
particular case of $m=2$ as well. We begin with the matrix 
\begin{equation}\label{G3}
G(n) = \begin{pmatrix}
\big ( \frac{(2\lambda-2)_n}{(1)_n} \big )^{1/2} & 0 & 0\\
(\frac{n}{2\lambda-2})^{1/2} \big (\frac{(2 \lambda
-1)_{n-1}}{(1)_{n-1}}\big )^{1/2}
 & \big (\frac{(2\lambda)_{n-1}}{(1)_{n-1}}\big )^{1/2} & 0\\
(\frac{n(n-1)}{(2\lambda-2)(2\lambda-1)} )^{1/2} \big
(\frac{(2\lambda)_{n-2}}{(1)_{n-2}} \big )^{1/2} &
 2 (\frac{n-1}{2\lambda})^{1/2} \big ( \frac{(2\lambda+1)_{n-2}}{(1)_{n-2}}
\big)^{1/2} & \big (\frac{ (2\lambda +2)_{n-2}}{(1)_{n-2}} \big )^{1/2}\\
\end{pmatrix}.
\end{equation}
The block  $W_n$ of the weighted shift $W$, in this case, is
\begin{equation}\label{W3}
\begin{pmatrix}
\big (\frac{n+1}{2\lambda + n -2}\big )^{1/2} & 0 & 0\\
\frac{-1}{\mu_1} \big (\frac{2\lambda-1}{2 \lambda - 2}\big )^{1/2}
\big (\frac{1}{(2\lambda+n-1)(2\lambda+n-2)}\big )^{1/2} &
\big (\frac{n}{2\lambda+n-1} \big )^{1/2} & 0\\
\frac{-2}{\mu_2}\big ( \frac{2\lambda+1}{(2\lambda -2)_3 } \big
)^{1/2} \big ( \frac{n}{(2\lambda+n-2)_3} \big )^{1/2} &
\frac{-2\mu_1}{\mu_2} \big (\frac{2\lambda+1}{2\lambda} \big )^{1/2}
\big (\frac{1}{(2 \lambda+n-1)(2\lambda+n)}\big )^{1/2} & \big
(\frac{n-1}{2\lambda+n}\big )^{1/2}
\end{pmatrix}.
\end{equation}
Finally, the reproducing kernel $K^{(\lambda,{\boldsymbol \mu})}$
with $m=2$ has the form:
\begin{eqnarray}\label{K3}
K^{(\lambda,{\boldsymbol \mu})}(z,w) &=&
\begin{pmatrix}
\frac{1}{(1-\bar{w}z)^{2\lambda-2}}&\frac{z}{(1-\bar{w}z)^{2\lambda-1}}&
\frac{z^2}{(1-\bar{w}z)^{2\lambda}} \\
\frac{\bar{w}}{(1-\bar{w}z)^{2\lambda-1}}&\frac{1+(2\lambda-2)\bar{w}z}
{(2\lambda-2)(1-\bar{w}z)^{2\lambda}}& \frac{z (2 +
(2\lambda-2)\bar{w}z)}
{(2\lambda - 2) (1-\bar{w}z)^{2\lambda + 1}}\\
{ \frac{\bar{w}^2}{(1-\bar{w}z)^{2\lambda}} }& \frac{\bar{w}(2 +
(2\lambda - 2)\bar{w}z)} {(2\lambda-2) (1-\bar{w}z)^{2\lambda+1}}&
\frac{2 + 4(2\lambda -1)\bar{w}z + (2\lambda - 1)(2\lambda -
2)z^2\bar{w}^2} {(2\lambda -1)(2\lambda-2)
(1-\bar{w}z)^{2\lambda+2}}
\end{pmatrix} \nonumber\\
&& + \,\mu_1^2 \begin{pmatrix}
0&0&0\\
0&\frac{1}{(1-\bar{w}z)^{2\lambda}}& 2 \frac{ z}
{(1-\bar{w}z)^{2\lambda +1}} \\
0&2 \frac{ \bar{w}} {(1-\bar{w}z)^{2\lambda+1}} &2
\frac{2}{2\lambda}
\frac{1+2\lambda\bar{w}z}{(1-\bar{w}z)^{2\lambda+2}} \\
\end{pmatrix} \nonumber\\
&& + \,\mu_2^2 \begin{pmatrix}
0&0&0\\
0&0&0\\
0&0&\frac{1}{(1-\bar{w}z)^{2\lambda+2}}\\
\end{pmatrix}.
\end{eqnarray}

\end{document}